\documentclass[a4paper,12pt]{article}

\usepackage{amsmath,amsfonts,amssymb,amsthm,dsfont,bbm,mathrsfs}
\usepackage{geometry}
\usepackage[hidelinks]{hyperref}
\usepackage{paralist}
\usepackage{enumitem}
\usepackage[english]{babel}
\usepackage{textcase} 
\usepackage{filemod}
\usepackage[longnamesfirst]{natbib}
\usepackage{etoolbox} 
\usepackage{color}

\makeatletter

\def\cite{\citet}

\theoremstyle{plain}
\newtheorem{theorem}{Theorem}[section]

\newtheorem{lemma}[theorem]{Lemma}
\newtheorem{corollary}[theorem]{Corollary}

\theoremstyle{definition}

\newtheorem{remark}[theorem]{Remark}

\numberwithin{equation}{section}
\allowdisplaybreaks

\def\@noindentfalse{\global\let\if@noindent\iffalse}
\def\@noindenttrue {\global\let\if@noindent\iftrue}
\def\@aftertheorem{%
  \@noindenttrue
  \everypar{%
    \if@noindent%
      \@noindentfalse\clubpenalty\@M\setbox\z@\lastbox%
    \else%
      \clubpenalty \@clubpenalty\everypar{}%
    \fi}}
\AfterEndEnvironment{theorem}{\@aftertheorem}
\AfterEndEnvironment{lemma}{\@aftertheorem}
\AfterEndEnvironment{definition}{\@aftertheorem}
\AfterEndEnvironment{lemma}{\@aftertheorem}
\AfterEndEnvironment{corollary}{\@aftertheorem}
\AfterEndEnvironment{remark}{\@aftertheorem}
\AfterEndEnvironment{example}{\@aftertheorem}
\AfterEndEnvironment{proof}{\@aftertheorem}

\renewcommand\section{\@startsection {section}{1}{\z@}%
{-3.5ex \@plus -1ex \@minus -.2ex}%
{1.3ex \@plus.2ex}%
{\center\small\sc\MakeTextUppercase}}

\def\subsection#1{\@startsection {subsection}{2}{0pt}%
{-3.5ex \@plus -1ex \@minus -.2ex}%
{1ex \@plus.2ex}%
{\bf\mathversion{bold}}{#1}}

\def\subsubsection#1{\@startsection{subsubsection}{3}{0pt}%
{\medskipamount}%
{-10pt}%
{\normalsize\itshape}{\kern-2.2ex. #1.}}

\def\be#1{\begin{equation*}#1\end{equation*}}

\def\bes#1{\begin{equation*}\begin{split}#1\end{split}\end{equation*}}

\def\ba#1{\begin{align*}#1\end{align*}}

\def\given{\mskip 1.5mu plus 0.25mu\vert\mskip 1.5mu plus 0.15mu}
\newcounter{bracketlevel}%
\def\@bracketfactory#1#2#3#4#5#6{%
\expandafter\def\csname#1\endcsname##1{%
\global\advance\c@bracketlevel 1\relax%
\global\expandafter\let\csname @middummy\alph{bracketlevel}\endcsname\given%
\global\def\given{\mskip#5\csname#4\endcsname\vert\mskip#6}\csname#4l\endcsname#2##1\csname#4r\endcsname#3%
\global\expandafter\let\expandafter\given\csname @middummy\alph{bracketlevel}\endcsname%
\global\advance\c@bracketlevel -1\relax%
}%
}
\def\bracketfactory#1#2#3{%
\@bracketfactory{#1}{#2}{#3}{relax}{0.5mu plus 0.25mu}{0.5mu plus 0.15mu}
\@bracketfactory{b#1}{#2}{#3}{big}{1mu plus 0.25mu minus 0.25mu}{0.6mu plus 0.15mu minus 0.15mu}
\@bracketfactory{bb#1}{#2}{#3}{Big}{2.4mu plus 0.8mu minus 0.8mu}{1.8mu plus 0.6mu minus 0.6mu}
\@bracketfactory{bbb#1}{#2}{#3}{bigg}{3.2mu plus 1mu minus 1mu}{2.4mu plus 0.75mu minus 0.75mu}
\@bracketfactory{bbbb#1}{#2}{#3}{Bigg}{4mu plus 1mu minus 1mu}{3mu plus 0.75mu minus 0.75mu}
}
\bracketfactory{clc}{\lbrace}{\rbrace}
\bracketfactory{clr}{(}{)}
\bracketfactory{cls}{[}{]}
\bracketfactory{abs}{\lvert}{\rvert}
\bracketfactory{norm}{\Vert}{\Vert}
\bracketfactory{floor}{\lfloor}{\rfloor}
\bracketfactory{ceil}{\lceil}{\rceil}
\bracketfactory{angle}{\langle}{\rangle}

\newcounter{ctr}\loop\stepcounter{ctr}\edef\X{\@Alph\c@ctr}%
	\expandafter\edef\csname s\X\endcsname{\noexpand\mathscr{\X}}
	\expandafter\edef\csname c\X\endcsname{\noexpand\mathcal{\X}}
	\expandafter\edef\csname b\X\endcsname{\noexpand\boldsymbol{\X}}
	\expandafter\edef\csname I\X\endcsname{\noexpand\mathbbm{\X}}
\ifnum\thectr<26\repeat

\let\@IE\IE\let\IE\undefined
\newcommand{\IE}{\mathop{{}\@IE}\mathopen{}}
\let\@IP\IP\let\IP\undefined
\newcommand{\IP}{\mathop{{}\@IP}}
\newcommand{\Var}{\mathop{\mathrm{Var}}}

\newcommand{\bigo}{\mathop{{}\mathrm{O}}\mathopen{}}
\newcommand{\lito}{\mathop{{}\mathrm{o}}\mathopen{}}
\newcommand{\law}{\mathop{{}\sL}\mathopen{}}

\let\original@left\left
\let\original@right\right
\renewcommand{\left}{\mathopen{}\mathclose\bgroup\original@left}
\renewcommand{\right}{\aftergroup\egroup\original@right}

\def\^#1{\relax\ifmmode {\mathaccent"705E #1} \else {\accent94 #1} \fi}
\def\~#1{\relax\ifmmode {\mathaccent"707E #1} \else {\accent"7E #1} \fi}
\def\*#1{\relax#1^\ast}
\edef\-#1{\relax\noexpand\ifmmode {\noexpand\bar{#1}} \noexpand\else \-#1\noexpand\fi}
\def\>#1{\vec{#1}}
\def\.#1{\dot{#1}}

\def\atop{\@@atop}
\def\%#1{\mathcal{#1}}

\renewcommand{\leq}{\leqslant}
\renewcommand{\geq}{\geqslant}
\renewcommand{\phi}{\varphi}

\newcommand{\D}{\Delta}
\newcommand{\N}{\mathop{{}\mathrm{N}}}

\newcommand{\I}{\mathop{{}\mathrm{I}}\mathopen{}}
\newcommand{\dtv}{\mathop{d_{\mathrm{TV}}}\mathopen{}}

\newcommand\indep{\protect\mathpalette{\protect\@indep}{\perp}}
\def\@indep#1#2{\mathrel{\rlap{$#1#2$}\mkern2mu{#1#2}}}

\newcommand{\toinf}{\to\infty}

\def\parsetime#1#2#3#4#5#6{#1#2:#3#4}
\def\parsedate#1:20#2#3#4#5#6#7#8+#9\empty{20#2#3-#4#5-#6#7 \parsetime #8}
\def\moddate{\expandafter\parsedate\pdffilemoddate{\jobname.tex}\empty}

\makeatother

\title{\sc\bf\large\MakeUppercase{Central limit theorems in the configuration model}}
\author{\sc A. D. Barbour \and \sc Adrian R\"ollin}

\date{\itshape Universit\"at Z\"urich and National University of Singapore}

\newcommand{\CM}{\mathop{\mathrm{CM}}}

\def\urn{\cU}

\def\ER{Erd\H{o}s-R\'enyi}

\def\setdiff{\triangle}

\def\la{\lambda}
\def\a{\alpha}
\def\m{\mu}
\def\t{\tau}
\def\s{\sigma}
\def\g{\gamma}
\def\ignore#1{}
\def\re{{\mathbb{R}}}

\def\HHH{H}
\def\llVert{\Vert}
\def\rrVert{\Vert}
\def\llvert{|}
\def\rrvert{|}

\begin{document}

\maketitle

\begin{abstract}
\noindent We prove a general normal approximation theorem for local graph
statistics in the configuration model,
together with an explicit bound on the error in the approximation with
respect to the Wasserstein metric.
Such statistics take the form~$T := \sum_{v \in V} H_v$, where~$V$ is
the vertex set, and~$H_v$ depends on a
neighbourhood in the graph around~$v$ of size at most~$\ell$.
The error bound is expressed in terms of~$\ell$, $|V|$, an almost sure
bound on~$H_v$,
the maximum vertex degree~$d_{\max}$ and the variance of~$T$. Under
suitable assumptions on the convergence of the empirical degree
distributions to a limiting distribution, we
deduce that the size of the giant component in the configuration model
has asymptotically Gaussian fluctuations.
\end{abstract}

\section{Introduction}\label{sec1}

Random graphs with a prescribed degree sequence have been intensively
studied in recent years.
This is largely because the binomial degree distributions that are
automatic in Erd\H os--R\'enyi random
graphs do not correspond well with those of many networks observed in
applications, making
more plausible null models essential for assessing statistical significance.
One of the first
results was obtained by \cite{Bender1978}, who investigated the total
number of
graphs with a prescribed degree sequence. Their result was later
generalised by \cite{Bollobas1980}, who made
use of the \emph{configuration model}; this is a random multigraph (i.e.,
a graph possibly containing loops and multiple edges), obtained from a
randomly chosen perfect matching
of the elements of the set of half-edges attached to each vertex. The
importance of the configuration model
lies in the fact that,
conditionally on there being no loops or multiple edges, the resulting
graph is distributed as a uniformly
chosen simple graph with the given degree sequence.

Under certain conditions on the degree sequence, \cite{Molloy1995}
showed that the configuration model has a
giant component which spans a fixed fraction of the vertices, in the
limit when the number of vertices tends
to infinity, with the degree distribution converging to a fixed
probability distribution.
However, relatively little progress has
been made towards understanding the \emph{second-order fluctuations}
of the size of the giant component. Only
recently, \cite{Ball2017} were able to take a first step in that
direction, by providing an asymptotic
expression for the variance of the size of the giant component,
and---by means of numerical simulations---conjecturing a central limit
theorem (CLT).

For the {\ER} random graph $G(n,p)$, CLTs are readily available for
many graph statistics; see, for
example, \cite{Rucinski1988} and \cite{Barbour1989} for subgraph
counts, \cite{Janson2008} for the so-called
\emph{susceptibility} and \cite{Pittel2005} and \cite{Bollobas2012}
for the size of the giant component, to
cite just a few examples. In contrast, for the configuration model, the
literature on CLTs is rather sparse;
we are only aware of the following CLTs: \cite{Janson2008b} for
the $k$-core, \cite{Angel2016} for the number
of loops and multiple edges in the infinite-variance case, \cite
{KhudaBukhsh2017} for certain statistics arising
from the SI-epidemic on the configuration graph and \cite{Riordan2012}
for the size of the giant component in the barely
supercritical case. Riordan's proof is quite different from ours. It is
based on a careful analysis of the exploration process,
and makes use of the graph being almost critical. Finally, \cite
{Athreya2018} have proved a CLT for additive summary statistics in the
subcritical case. Their proof is based on the martingale CLT and,
therefore, also quite different from ours.

The purpose of this article is to prove a general CLT for statistics
which can be expressed as sums of ``local'' vertex statistics. By
``local'', we mean statistics that are determined by
a limited part of the neighbourhood of a vertex that is close to the
vertex itself---for instance, by
the $\ell$ closest neighbours with respect to graph distance, where a
suitable rule is used
if necessary to choose among neighbours that are at the same distance
from the vertex.
Then, as an application of the general result, we prove a CLT for the
size of
the giant component in the configuration model. Following the strategy
of \cite{Ball2017}, instead of approximating
the size of the giant component directly, we approximate the number of
vertices in components of size at
most $n^{\beta}$, for some small $\beta>0$, which is asymptotically
equivalent to the number of vertices \emph{not} in
the giant component.

Let $d=(d_1, \dots, d_n)$ be such that $m:=\sum_{v=1}^n d_v$ is even,
a ``vector of degrees'',
and write $d_{\max}:= \max_{1\leq v\leq n}d_v$.
Let $\cG\sim\CM(d)$ be a realisation of the configuration multigraph
on $n$ vertices, labelled $1$ through $n$,
where vertex $v$ has degree $d_v$.
For each vertex $v \in[n] := \{1,2,\ldots,n\}$, let $\cT(v)$ denote
the rooted component in $\cG$ containing $v$,
with $v$ being assigned the root label, making~$\cT(v)$ a finite,
connected, vertex-labelled and rooted multigraph.

Let $h(\cT)$ be a real-valued function on finite, connected,
vertex-labelled and rooted multigraphs $\cT$.
We are interested in the random quantity
%
\begin{equation}
\label{1} U_{d,h} = \sum_{v\in[n]}h\bigl(
\cT(v)\bigr),
\end{equation}
and the corresponding centred and normalised version $\hat{U}_{d,h} =
(U_{d,h}-\mu_{d,h})/{\sigma_{d,h}}$,
where $\mu_{d,h}=\IE U_{d,h}$ and $\sigma_{d,h}^2 = \Var U_{d,h}$.

\subsection{A central limit theorem for local graph statistics}

For two distributions $F$ and $G$ on $\re$, denote by $d_{\mathrm{W}}(F,G)$
the Wasserstein distance
between $F$ and $G$; that is,
$d_{\mathrm{W}}(F,G) := \sup_{f \in\cF_{\mathrm{W}}} |\int f dF -
\int
f dG|$, where $\cF_{\mathrm{W}}$ denotes the set of real valued
functions with Lipschitz constant at most $1$.
Write $\N(0,1)$ for the standard normal distribution. For any
multigraph $\cT$, we use $ \llvert \cT \rrvert $
to denote the number of vertices in $\cT$.

Let $h$ be a nonnegative function on finite, connected, vertex-labelled
and rooted multigraphs, and
let $\ell\geq1$. We say that \emph{$h$ only depends on the
first $\ell$ vertices away from the root},
if the following holds. Let $\cT$ and $\cT'$ be two finite,
connected, vertex-labelled and rooted multigraphs.
Let $\cT_\ell$ and $\cT'_\ell$ be the respective subgraphs induced
by exploring the respective multigraphs,
starting with the root, breadth first and, for each wave of
exploration, smallest label first, until at
most $\ell$ vertices (including the root) have been explored. If $\cT
_\ell= \cT'_\ell$,
then $h(\cT_\ell) = h(\cT'_\ell)$.

\begin{theorem}\label{THM1} Assume that $h$ only depends on the
first $\ell$ vertices
away from the root. Then, if $2 \leq d_{\max}\leq n^{1/4}$,
$12 \leq\ell\leq n^{1/4}$ and $m\geq\max\{n, 7d_{\max}^2\ell\}$, it
follows that
\begin{equation*}
d_{\mathrm{W}}\bigl(\law(\hat{U}_{d,h}),\N(0,1)\bigr)
\leq
\frac{\llVert  h
 \rrVert ^3d_{\max}^2\ell^{10} n}{4536\sigma_{d,h}^3} + \frac{\llVert  h  \rrVert ^2d_{\max}^{2}\ell^{8}
n^{1/2}}{78\sigma_{d,h}^2},
\end{equation*}
where $ \llVert  h  \rrVert $ denotes supremum norm.
\end{theorem}

Theorem~\ref{THM1} can potentially be applied to many graph
statistics: sub-graph counts, the number of small
components and the susceptibility, to name but a few. The upper bounds
on $d_{\max}$ and $\ell$ are largely
unimportant, since the bound given in Theorem~\ref{THM1} would
typically be larger than $1$ if they were
violated. Evaluating the order of the approximation error depends on
obtaining lower bounds on the
variance $\sigma_{d,h}^2$. For the size of the giant component, it is
of strict order $n$; see also
the discussion in Section~\ref{sect-variance}.

\subsection{Size of the giant component}

We now use Theorem~\ref{THM1} to prove our second main result.
For each $n\geq1$, let $d^{(n)}=(d^{(n)}_1,\dots,d^{(n)}_n)$ be
a vector of degrees on $n$ vertices;
let $\pi= (\pi_j)_{j\geq0}$ be a probability distribution on the
nonnegative integers. We denote
by $\pi^{(n)} :=  \{n^{-1}\sum_{i=1}^n \I[d^{(n)}_i=j],
0\leq j\leq n  \}$ the empirical degree distribution
of $d^{(n)}$. Let the random variable $D_n$ have distribution $\pi
^{(n)}$, and let $D$ have
distribution $\pi$. Denote by $D^*$ a random variable having the
size-bias distribution of $D$, that is,
$\IP[D^*=j]=j\pi_j/\IE D$, $j\geq1$, and note that $\IE D^* = \IE
D^*/\IE D$. The following is the main
result of \cite{Ball2017}; all limits are to be understood as $n\toinf$.

\begin{theorem}[{\cite[Theorem~2.1]{Ball2017}}] \label{thm2} Assume that
\ba{
  (1)\enskip& \IE D^* > 2,&
  (2)\enskip& \pi_1>0,&
  (3)\enskip& \IE D^3 <\infty,\qquad\qquad\\[1ex]
  (4)\enskip&\law(D_n)\to\law(D), &
  (5)\enskip&\IE D_n^3\to \IE D^3;\\[1ex]
  (6)\enskip&\rlap{\begin{minipage}[t]{0.85\textwidth}there exists~$\beta>0$ such 
             that~$\dtv\bclr{\law(D^*_n),\law(D^*)}=\bigo\clr{n^{-\beta}}$ and such 
                   that~$\abs{\IE D_n-\IE D} = \bigo\clr{n^{-\beta}}$;\end{minipage}}\\[1ex]
  (7)\enskip&\rlap{\begin{minipage}[t]{0.8\textwidth}there exists~$\delta>0$ such 
             that~$d^{(n)}_{\max} = \bigo(n^{1/4-\delta})$.\end{minipage}}
}
Let~$R_n$ denote the size of the largest component of\/~$\CM\bclr{d^{(n)}}$. Then there exists a positive 
constant~$b^2$ such that, as~$n\toinf$,
\be{
  \Var R_n \sim nb^2.
}
\end{theorem}

\begin{remark} \label{rem1}
Condition (1) is equivalent to the usual threshold condition of \cite
{Molloy1995}, which guarantees the
existence of a giant component with high probability. Condition (2)
ensures that,
under Conditions (3), (4) and~(5), with
high probability not all vertices are either isolated or part of the
same giant component;
if $\pi_1=0$, $\Var R_n$ may be of smaller asymptotic order than $n$.
Condition (3) is
the moment condition of \cite{Ball2017}; Condition (4) is the
condition that the empirical degree
distribution converges; Conditions (5) and (6), respectively, are
equivalent to Conditions (a)(ii) and (b)
of \cite{Ball2017}; Condition (7) is Condition (c) of \cite
{Ball2017}. Moreover, it is straightforward
to check that Conditions~(3), (4) and (5) imply Condition (a)(i) of
\cite{Ball2017}, which is, as is easily
verified, redundant there. Conditions (3), (5), (6) and (7) are
required in their current forms in the
proof given by \cite{Ball2017}, and we use their variance asymptotics
in proving Theorem~\ref{thm3}. We also
need something close to Condition (7) to ensure that the bound in
Theorem~\ref{THM1} is small.
However, the conditions could possibly be relaxed somewhat.
\end{remark}

\begin{theorem} \label{thm3} Assume that the conditions of
Theorem~\ref{thm2} hold.
Let $\hat{R}_{n}$ denote the centred and normalized version of
$R_{n}$. Then
\begin{equation*}
\law(\hat{R}_{n})\to N(0,1) \qquad\text{as $n\toinf$}.
\end{equation*}
\end{theorem}

\begin{proof}
As mentioned in the \hyperref[sec1]{Introduction}, instead of counting the number of
vertices in the giant component, we proceed as
\cite{Ball2017} and count the number of vertices that are not in
the giant component, which is
just $n-R_{n}=:S_{n}$; a CLT for $S_n$ obviously implies a CLT
for $R_n$. Let $h_\ell(\cC)$ be the function
that equals~$1$ if~$ \llvert \cC \rrvert \leq\ell$ and
equals $0$ otherwise;
hence, $U_n:= U_{d^{(n)},h_\ell}$ from \eqref{1}
is the number of vertices in components of size less or equal to $\ell
$ in~$\cG\sim\CM(d^{(n)})$. Under the conditions of
Theorem~\ref{thm3}, it follows from the proof of \cite{Molloy1995},
Lemma~11 (see also the discussion after
Theorem~2.1 of \cite{Ball2017}) that, for any fixed $\delta> 0$,
the probability that a vertex lies in a component larger
than $n^{\delta/10}=:\ell$,
but not in the largest component, is bounded by $Cn^{-2}$, for some
constant $C>0$ independent of $n$; hence,
%
\begin{equation}
\label{2} \IP[S_{n} \neq U_{n}] \leq\frac{C}{n}.
\end{equation}
In what follows, $C$ will denote a generic constant that may differ
from line to line, but is always independent of $n$.
Let $\lambda_{n} := \IE S_{n}$ and let $\tau^2_{n} := \Var S_{n}$, and with~$\mu_n=\IE U_n$ and $\sigma^2_n = \Var U_n$, let
\begin{equation*}
a_n = \frac{\m_{n}-\la_{n}}{\t_{n}}, \qquad b_n^2 =
\frac{\s_{n}^2}{\t_{n}^2}.
\end{equation*}
Since, under the conditions of Theorem~\ref{thm3}, $\sigma_{n}^2\sim
\tau_{n}^2$ (see the discussion after
Theorem~2.1 of \cite{Ball2017}) and since \eqref{2} implies
that $ \llvert \lambda_{n}-\mu_{n}  \rrvert =\bigo(1)$, we have
\begin{equation*}
\llvert a_n \rrvert \leq\frac{C}{\sqrt{n}}, \qquad \bigl\llvert
1-b^2_n \bigr\rrvert = \lito(1).
\end{equation*}
It is a standard exercise in Stein's method for normal approximation to show
that~$d_{\mathrm{W}}(\N(0,1),\N(a,b^2))\leq \llvert  a  \rrvert + \llvert 1-b^2  \rrvert $.
Now
\begin{equation*}
\frac{U_n - \la_n}{\t_n} = {b_n} \biggl(\frac{U_n-\m_n}{\s
_n} \biggr) +
a_n,
\end{equation*}
so that, from the definition of $d_{\mathrm{W}}$,
\begin{equation*}
d_{\mathrm{W}}\bigl(\law\bigl(\t_n^{-1}(U_n-
\la_n)\bigr),\N\bigl(a_n,b_n^2
\bigr)\bigr) = b_n d_{\mathrm{W}}\bigl(\law\bigl(
\s_n^{-1}(U_n-\m_n)\bigr),\N(0,1)
\bigr).
\end{equation*}
Now, let $h$ be a real valued function bounded in modulus by 1 and with
Lipschitz constant at most 1, and
let $Z$ have a standard normal distribution. We have
\begin{align*}
& \bigl\llvert \IE h(\hat{S}_{n}) - \IE h(Z) \bigr\rrvert
\\
& \qquad\leq 2\IP[S_{n} \neq U_{n}] + d_{\mathrm{W}}
\bigl(\law\bigl(\t_n^{-1}(U_n-\la_n)
\bigr),\N\bigl(a_n,b_n^2\bigr)\bigr)
\\
&\qquad\quad{} + d_{\mathrm{W}}\bigl(\N\bigl(a_n,b_n^2
\bigr),\N(0,1)\bigr)
\\
&\qquad \leq C\biggl(\frac{1}{n}+\biggl(\frac{d_{\max}^2\ell
^{10}}{n^{1/2}}+
\frac{d_{\max}^{2}\ell^{8}}{n^{1/2}}\biggr)\biggr) +\lito(1) = \lito(1). %
\end{align*}
Since the set of bounded and Lipschitz continuous functions
characterises convergence in distribution, the claim follows.
\end{proof}

\begin{remark} By conditioning on the multigraph to be simple, results
obtained for the configuration model can
often be transferred to simple graphs (see, e.g., \cite{Janson2010b}).
However, as pointed out
by \cite{Janson2010b}, Remark~1.4, distributional limit results cannot
be transferred in general, so the
question about the fluctuations of the size of the giant component in
simple graphs with a prescribed degree
sequence remains open.
\end{remark}

\section{Technical preliminaries}

We begin by discussing some technical material as a preparation for the
proof of Theorem~\ref{THM1}, which
is presented in the next section. Our main tool is Stein's method for
normal approximation (see \cite{Stein1972}
and \cite{Chen2011}). In particular, we make use of \emph{Stein
couplings}, for which we refer
to \cite{Chen2010b} for a detailed discussion.

\subsection{An abstract normal approximation theorem}

We say that a triple of random variables $(W,W',G)$ is a \emph{Stein
coupling} if
%
\begin{equation}
\label{3} \IE\bigl\{Gf\bigl(W'\bigr)-Gf(W)\bigr\} = \IE\bigl
\{Wf(W)\bigr\}
\end{equation}
for all functions $f$ for which the expectations exists. The proof of
the following result is standard;
see also \cite{Chen2010b}, Corollary~2.2.

\begin{theorem}\label{thm4}
Let $(W,W',G)$ be a Stein coupling with $\Var W = 1$. Then
\begin{equation*}
d_{\mathrm{W}}\bigl(\law(W),\N(0,1)\bigr)\leq0.8\sqrt{\Var\IE(G\D| W)} +
\IE \bigl\llvert G\D^2 \bigr\rrvert ,
\end{equation*}
where $\D=W'-W$.
\end{theorem}

Since Stein couplings are rather abstract, we now present a general
construction that leads to the Stein coupling
that is used in the proof of Theorem~\ref{THM1}. It is not clear how
other couplings that have appeared
in the literature could have been applied. For instance, Stein's
exchangeable pair coupling (see \cite{Stein1986})
requires a certain
linearity condition to be satisfied, which would be difficult to verify
here. The flexibility offered by
Stein couplings is thus important. Moreover, the following
construction, which is a variant of
Construction 2A of \cite{Chen2010b}, reduces the effort needed in
deriving the bounds,
because the $X_i$ appearing in the construction need not be centred.

\begin{lemma}\label{lem1}
Let $X_1,\dots,X_n$ be random variables and let $U=\sum_{i=1}^n X_i$.
Let~$W = \sigma^{-1}(U-\IE U)$,
where $\sigma^2:=\Var U < \infty$. Assume that, for each $i$, we can
construct a random variable $W'_i$ such that $X_i$
and $W'_i$ are independent and such that $\law(W'_i) = \law(W)$.
Let $I$ be a random variable uniformly distributed
on~$[n]$ and independent of all else. Let
\begin{equation*}
G_i = -\frac{n}{\sigma}X_i,\qquad G =
G_I,\qquad W' = W'_I.
\end{equation*}
Then $(W,W',G)$ is a Stein coupling.
\end{lemma}
\begin{proof}
Indeed,
\begin{equation*}
\IE\bigl\{Gf(W)\bigr\} = -\frac{1}{\sigma}\sum_{i=1}^n
\IE\bigl\{X_if(W)\bigr\}
\end{equation*}
and
\begin{equation*}
\IE\bigl\{Gf\bigl(W'\bigr)\bigr\} = -\frac{1}{\sigma}\sum
_{i=1}^n \IE\bigl\{X_i f
\bigl(W'_i\bigr)\bigr\} = -\frac{1}{\sigma}\sum
_{i=1}^n \IE X_i \IE f(W),
\end{equation*}
so that
\begin{equation*}
\IE\bigl\{Gf\bigl(W'\bigr)-Gf(W)\bigr\} = \frac{1}{\sigma}\sum
_{i=1}^n \IE\bigl\{(X_i - \IE
X_i)f(W)\bigr\} = \IE\bigl\{Wf(W)\bigr\},
\end{equation*}
establishing \eqref{3}.
\end{proof}

Broadly speaking, the difference $W-W'_i$ measures how much the sum $W$
needs to be changed in order to become
independent of its summand $X_i$. If this difference is small, the
influence of $X_i$ on $W$ is weak. Randomizing
over the summands, we can similarly argue that if $W-W'_I$ is small,
then the influence of most $X_i$ on $W$ is weak,
indicating that $W$ is a sum of weakly dependent random variables and,
therefore, making a CLT for $W$ plausible.

\subsection{Configurations}

Instead of working with multigraphs, we follow the standard procedure
and work instead with matchings of
coloured balls, where the colours represent the individual vertices,
and the balls represent the
half-edges coming out of each vertex. For simplicity, we represent the
colours by the numbers in $[n]$,
and we represent the labels of the balls by the numbers in $[m]$.
Let~$d =(d_i)_{1\leq i \leq n}$ be a
vector of degrees, so that $m = \sum_i d_i$ is even. One may think of
balls $1$
to $d_1$ having colour 1, balls $d_1+1$ to $d_1+d_2$ having colour 2,
and so forth. A~configuration $\cG$ on
this set of balls is simply a perfect matching of the balls (note that
a perfect matching consists of an unordered set
of $m/2$ nonoverlapping pairs); denote by $\mathcal{M}_{0}(d)$ the set
of all
such perfect matchings,
and denote by $\cU_{0}(d)$ the
uniform distribution on $\mathcal{M}_{0}(d)$ (the reason for the ``0''
in the
notation becomes clear in the next paragraph).
Now, in this notation, saying that $\cG$ is a configuration random
graph $\CM(d)$
is equivalent to saying that $\cG\sim\cU_0(d)$.

We also need to consider sub-configurations of $\cG$, and these may
contain \emph{unpaired} balls.
To make this precise, let $C\subset[n]$ be a subset of colours. We
denote by $\cG|_C$ the sub-configuration
of $\cG$ restricted to the set of balls which have any of the colours
in $C$. This comprises the pairs
for which the colours of \emph{both} balls are in $C$, together with
unpaired balls, with colours in $C$,
whose partners in~$\cG$ have colour in $C^c$.
Let $s(C)$ be the number of unpaired
balls in $\cG|_C$; clearly, $s(C)=s(C^c)$, since in the base
configuration $\cG$ all balls are matched. Moreover,
denote by $d(C) = (d_1(C),\dots,d_n(C))$ the degree sequence of
length $n$, now restricted to the colours in $C$,
so that $d_v(C) = d_v$ if $v\in C$ and $d_v(C) = 0$ if $v\notin C$.
Finally, for any fixed $s$, let $\mathcal{M}_{s}(d(C))$
denote the set of matchings of the balls having colours in $C$,
with $s$ balls left unpaired, and denote
by $\cU_s(d(C))$ the uniform distribution on $\mathcal{M}_{s}(d(C))$.

The following lemma is a consequence of the uniform distribution on matchings.

\begin{lemma}\label{lem2} Let $C\subset[n]$ be a fixed subset of
colours, and let $\cG\sim\urn_0(d)$.
Then, conditionally on $s(C)$, we have that
$\cG|_C$ and $\cG|_{C^c}$ are independent with
%
\begin{equation}
\label{4} \cG|_C \sim\urn_{s(C)}\bigl(d(C)\bigr) \quad
\text{and}\quad \cG|_{C^c} \sim\urn_{s(C)}\bigl(d
\bigl(C^c\bigr)\bigr).
\end{equation}
\end{lemma}

In other words, the lemma says that, given $s(C)$, the two
sub-configurations $\cG|_C$ and $\cG|_{C^c}$
are independent and themselves uniformly distributed, but on different
vectors of degrees,
and with unpaired balls if $s(C)>0$.

\subsection{Truncated components}

The key assumption in Theorem~\ref{THM1} is that the function $h$ only
depends on the first $\ell$ vertices away
from the root. It is therefore enough to explore the component for each
vertex breadth first, up to the point
where $\ell$ vertices have been explored---we denote this truncated
component by~$\cT_\ell(v)$.

\bigskip

\noindent\textit{Algorithm A}

\begin{enumerate}
\item Let $j\leftarrow1$, let $C_1 \leftarrow\{v\}$, and let $\cS_1
\leftarrow\cG|_{\{v\}}$.
\item\label{algorithmb1} If $j=\ell$ or if the
sub-configuration $\cS_j$ has no unpaired balls, then
proceed to Step~\ref{algorithmb4}.
\item\label{algorithmb2} Among all the unpaired balls in $\cS_j$,
take the one with the \emph{smallest ball label}
from the colour with the \emph{smallest colour label} among the
colours closest to $v$, reveal its partner,
and denote by $w$ the colour of that partner.
\item Let $j\leftarrow j+1$.
\item\label{algorithmb3} Let $C_j \leftarrow C_{j-1}\cup\{w\}$, and
let $\cS_j \leftarrow\cG|_{C_j}$.
\item Return to Step~\ref{algorithmb1}.
\item\label{algorithmb4} Let $\cT_\ell(v)\leftarrow\cS_{j}$, and stop.
\end{enumerate}
%

%

\begin{remark} A few comments are in order.
\begin{enumerate}
\item With $ \llvert \cT_\ell(v)  \rrvert $ denoting the
number of colours
in $\cT_\ell(v)$, we have $ \llvert \cT_\ell(v)  \rrvert
\leq\ell$.
\item At any stage $j$ in the algorithm, each of the unpaired balls
in $\cS_j$ is paired (from the viewpoint of $\cG$)
with a colour not yet in $\cS_j$, since by the definition of
sub-configurations, once a new colour is added, all loops
and all pairings with colours already in the configuration are
automatically revealed.
\item It is not difficult to see that $\cT_\ell(v)$ contains unpaired
balls if and only if $\cT(v)$ contains
strictly more than $\ell$ vertices.
\item Under the assumptions on $h$, it is clear that $h(\cT(v)) =
h(\cT_\ell(v))$; hence
\begin{equation*}
U_{d,h} = \sum_{v\in[n]} h\bigl(
\cT_\ell(v)\bigr).
\end{equation*}
\end{enumerate}
\end{remark}

\begin{remark}\label{rem20} We need to be able to condition the
configuration $\cG$ on the realization of $\cT_\ell(v)$,
or rather on the set of colours in $\cT_\ell(v)$ along with the
number of unpaired balls in $\cT_\ell(v)$. To shorten notation, we
let $\xi_v$ denote the set of colours contained in $\cT_\ell(v)$.
Then, in order to construct the couplings of Lemma~\ref{lem1}, we
shall use statements of the form
%
\begin{equation}
\label{5} \law\bigl(\cG|_{\xi_v^c}\big|\xi_v,s(
\xi_v)\bigr) = \urn_{s(\xi
_v)}\bigl(d\bigl(\xi_v^c
\bigr)\bigr),
\end{equation}
which is the analogue of the strong Markov property, but which does not
follow immediately
from Lemma~\ref{lem2}. However, \eqref{5} can be rigorously
established, and has in fact been
used implicitly in the literature many times. We refer to \cite
{Rozanov1982} for the general theory of
\emph{stopping sets}; to establish \eqref{5}, we refer in particular
to \cite{Rozanov1982}, Lemma~1, p. 75.
\end{remark}

\subsection{Stein coupling}

We proceed to construct the Stein coupling necessary for the proof of
the main theorem. In order to
shorten notation, we consider $n$, $\ell$ and the degree sequence $d$
to be fixed, and throughout the remainder
of the article, we let $\xi_v$ denote the set of colours contained in
$\cT_\ell(v)$ as before; in particular, note
that~$ \llvert \xi_v  \rrvert \leq\ell$.

Let $\cG\sim\urn_0(d)$. Let $X_v = h(\cT_\ell(v))$ for
every $v\in[n]$, and let $W=\sum_{v\in[n]} X_v$.
Consider $v$ fixed for now. Starting from $\cG$, we construct a new
configuration $\cG'_v$ in such a
way that $\cG_v'$ is independent of $\cT_\ell(v)$ and, therefore,
independent of $X_v$.
We establish independence by showing that the conditional distribution
of $\cG'_v$, given~$\cT_\ell(v)$, does not
depend on $\cT_\ell(v)$; specifically,  we now show that we have~$\law(\cG
_v'|\cT_\ell(v))=\law(\cG)$.

We start with the sub-configuration $\cG|_{\xi_v^c}$. By \eqref{5},
$\cG|_{\xi_v^c}$ contains $s(\xi_v)$ unpaired balls,
which, conditionally on $\cT_\ell(v)$, are uniformly distributed
among the balls in~$\cG|_{\xi_v^c}$. We can thus add
all unpaired balls from $\cG|_{\xi_v}$ into $\cG|_{\xi_v^c}$, and
pair them with the partners that
they were already paired with in $\cG$; denote the resulting
sub-configuration by $\cG'_{v,0}$.

In order to add the remaining balls to $\cG'_{v,0}$, we proceed
step-wise, one pair at a time.
In what follows, we will make random choices; so denote by $B_v$ a
suitable source of random numbers
(e.g., $B_v$ could just be a sequence of independent uniform random
variables), in such a way that the choices
made are a deterministic function of $B_v$.

Let $K_v$ denote the number of pairings among the remaining balls, and,
if~\mbox{$K_v>0$}, repeat the following
procedure $K_v$ times (if $K_v=0$, there is nothing to be done). The
result is a
sequence $\cG'_{v,0},\cG'_{v,1},\dots,\cG'_{v,K_v}$ with the
property that the matchings
in $\cG'_{v,k}$ are uniformly distributed among all matchings between
the respective balls;
in particular, $\cG'_{v,K_v}\sim\cU_0(d)$, irrespective of $\cT
_\ell(v)$.

\paragraph{Constructing $\boldsymbol{{\cG'_{v,k}}}$
from $\boldsymbol{\cG'_{v,k-1}}$ for $\boldsymbol{1\leq k\leq K_v}$.}
Among the balls not in $\cG'_{v,k-1}$, pick a pair of matched balls.
Independently of all else, toss a coin that shows heads with
probability $1/(m-2(K_v-k)-1)$ and that shows tails with
probability \mbox{$(m-2(K_v-k)-2)/(m-2(K_v-k)-1)$}. If the coin shows heads,
add the two balls to $\cG'_{v,k-1}$, match them with each other and
denote the resulting configuration by $\cG'_{v,k}$.
If the coin shows tails, pick one ball from $\cG'_{v,k-1}$ uniformly
at random, call it $I$ and call its partner~$J$,
and break up the bond between $I$ and $J$. Add one of the balls to be
added to~$\cG'_{v,k-1}$ and match it with $I$,
and then add the remaining ball to $\cG'_{v,k-1}$ and match it
with~$J$; denote the resulting configuration by $\cG'_{v,k}$.

It is not difficult to convince oneself that the matchings in
each $\cG'_{v,k}$ are uniformly distributed.
Assuming that they are for $\cG'_{v,k-1}$, we add a pair of balls, and
leave them paired with probability $1/(m-2(K_v-k)-1)$;
this is exactly the probability that two randomly chosen balls are
matched in $\cG'_{v,k}$. Otherwise, we match each of the two
balls with a ball randomly chosen from $\cG'_{v,k-1}$. The way to do
this is to pick a
random pair from $\cG'_{v,k-1}$, break it up, and pair the two balls
individually with the balls that have just been added.

We need to keep track of the colours involved in the construction.
Let $\HHH_v$ denote the union of the colours of the balls $I$ and $J$
picked in the steps $1\leq k\leq K_v$; that is,~$\HHH_v$ consists of the colours of all the balls used in
applying the above construction. Then define $\eta_v:= \xi_v \cup
\HHH_v$.
Note that $\xi_v$ is a deterministic function of~$\cG$, whereas $\eta
_v$ is a
deterministic function of both $\cG$ and $B_v$; moreover, \eqref{5}
also holds if $\xi_v$ is replaced
by $\eta_v$ throughout.

Finally, let $\cG_v' := \cG'_{v,K_v}$, which is the configuration
obtained after all balls have been put back. As mentioned before,
conditionally on $\cT_\ell(v)$, $\cG_v'$ is distributed as~$\urn_0(d)$;
hence, $\cG_v'$ is independent of $\cT_\ell(v)$, and hence,
independent of $X_v$, since the latter is a function of $\cT_\ell(v)$.

Now, for each $w\in[n]$, let $\cT^v_\ell(w)$ denote the truncated
component of $w$ in $\cG_v'$, that is,
the truncated component obtained by applying Algorithm A in $\cG_v'$,
starting with $w$; let
\begin{equation*}
W_v' = \frac{1}{\sigma_{d,h}} \Biggl(\sum
_{w=1}^n h\bigl(\cT^v_\ell(w)
\bigr)-\mu_{d,h}\Biggr).
\end{equation*}
Since $W_v'$ is independent of $X_v$ and has the same distribution
as $W$, we apply Lemma~\ref{lem1} to obtain a Stein coupling
$(W,W',G)$.

\section{Proof of Theorem \protect\ref{THM1}}

As a first step, we show that $K_v-\ell$ is bounded with high
probability, uniformly in $v\in[n]$,
as long as $d_{\max}\ell$ grows no faster than a small power of $m$.
\begin{lemma}\label{K_v-bound}
For $k\geq1$ and $m \geq8(k\vee \ell)$,
\begin{equation*}
\IP[K_v \geq\ell+k-1] \leq\frac{d_{\max}^{2k}\ell^{2k}}{k!m^k}.
\end{equation*}
In particular, if $m \geq n$,
\begin{equation*}
\IP\bigl[K_v \geq\ell+ k - 1\text{ for some }v\in[n] \bigr] \leq
\frac{d_{\max}^{2k}\ell^{2k}}{k!m^{k-1}}.
\end{equation*}
\end{lemma}
\begin{proof}At Step~3,
Algorithm A reveals the partner of an unmatched ball; there are at
most $\ell-1$
of pairings revealed in this manner. In addition, at Step~5, other
pairs may be revealed; if $w$ denotes the colour of the
partner revealed at Step~3, then some of its remaining $d_w-1$ balls
may be paired with unpaired
balls having the previously chosen colours. At any stage, there are no
more than $ d_{\max}\ell$
unpaired balls. The chance of any two given balls being paired
is at most $1/(m-2\ell+1)$, and the chance that any $k$
given sets of two balls are each paired is $ \{\prod_{j=1}^k
(m-2\ell-2j+3) \}^{-1}$. Hence the expected number of
$k$-tuples of matched pairs in $\cT_\ell(v)$, other than those
revealed at Step~3, is at most
\begin{equation*}
\begin{split}
&\binom{\binom{d_{\max}\ell}{2}} {k} \frac{1}{\prod_{j=1}^k (m-2\ell-2j+3)} \\
&\qquad\leq \binom{d_{\max}\ell}{2}^k\frac{1}{k!\prod_{j=1}^k (m-2\ell-2j+3)}\\
&\qquad\leq \frac{d_{\max}^{2k}\ell^{2k}}{k!2^k\prod_{j=1}^k (m-2\ell-2j+3)}\\
&\qquad\leq \frac{d_{\max}^{2k}\ell^{2k}}{k!m^k 2^k\prod_{j=1}^k (1-2\ell/m-2j/m+3/m)}
        \leq \frac{d_{\max}^{2k}\ell^{2k}}{k!m^k};
\end{split}
\end{equation*}
the fact that $2^k\prod_{j=1}^k (1-2\ell/m-2j/m+3/m)\geq 1$ follows from the fact that $(1-2\ell/m-2j/m+3/m)\geq 1/2$ under the given assumptions. The two claims now easily follow.
\end{proof}

In what follows, we define the event $A$ by
%
\begin{equation}
\label{A_n} A := \bigl\{K_v \leq\ell+ 6\text{ for all }v\in[n]
\bigr\},
\end{equation}
and write
%
\begin{equation}
\label{gamma_n} \g:= \IP\bigl[A^c\bigr] \leq\frac{d_{\max}^{16}\ell^{16}}{8!m^7},
\end{equation}
where the upper bound in \eqref{gamma_n} follows directly from
Lemma~\ref{K_v-bound} with $k=8$.

\subsection{Bounds on intersection probabilities}

\begin{lemma}\label{lem3} Let $\cG\sim\cU_0(d)$, and assume that a
random set of colours $\alpha$ has been
obtained, perhaps using $\cG$, but in such a way that
\begin{equation*}
\law\bigl(\cG|_{\alpha^c}\big|\alpha,s(\alpha)\bigr) = \cU _{s(\alpha)}
\bigl(d\bigl(\alpha^c\bigr)\bigr).
\end{equation*}
Then, for every $v\in[n]$,
%
\begin{align}
\IP\bigl[\xi_v \cap\alpha\neq\emptyset\big|\alpha,s(\alpha)\bigr] &
\leq I_{v\in\alpha}+\frac{2d_{\max}\llvert\alpha\rrvert(\ell-1)}{m} \label{6}
\\
\intertext{and} \IE\bigl\{ \llvert \xi_v \cap\alpha \rrvert \big|\a
,s(\alpha )\bigr\} & \leq \ell I_{v\in\alpha} + \frac{2d_{\max}\llvert\alpha\rrvert(\ell-1)}{m},
\label{6.5}
\end{align}
whenever $m\geq2d_{\max}( \llvert \alpha \rrvert +\ell)$.
Similarly, we have
%
\begin{align}
\IP\bigl[\eta_v\cap\alpha\neq\emptyset,A\big|\alpha,s(\alpha)\bigr]
& \leq I_{v\in\alpha}+\frac{2d_{\max}\llvert\alpha\rrvert(3\ell+11)}{m}, \label
{7}
\\
\intertext{and} \IE\bigl\{ \llvert \eta_v \cap\alpha \rrvert
I_A \big|\a ,s(\alpha)\bigr\} & \leq (3\ell+ 12) I_{v\in\alpha} +
\frac{2d_{\max}\llvert\alpha\rrvert(3\ell+11)}{m}. \label{7.5}
\end{align}
\end{lemma}

\begin{proof} Note that the event $\{\xi_v\cap\alpha\neq\emptyset\}
$ happens either if $v\in\alpha$ or
if $v\notin\alpha$ and, during the exploration of $\cT_\ell(v)$, a
ball whose partner is to be revealed in
Step~\ref{algorithmb2} of Algorithm A is an unpaired ball in $\cG
|_{\alpha^c}$, since then the partner is of a
colour in $\alpha$. If $v\in\alpha$, the bound is trivial; so,
assume $v\notin\alpha$. Before the exploration
starts, colour~$v$ is already considered explored, and since there
could be loops, the pairings of
up to $d_{\max}$ balls could have been revealed. If there are still
unpaired balls at this phase, the exploration
process starts. At this point, the probability that a specific ball is
unpaired in $\cG|_{\alpha^c}$ is
at most $s(\alpha)/(m-d_{\max} \llvert \alpha \rrvert
-d_{\max})$, so the
first ball
whose partner is to be revealed has at most
this probability of being an unpaired ball in $\cG|_{\alpha^c}$. If
the process continues, the next ball
whose partner is to be revealed has
a probability of at most $s(\alpha)/(m-d_{\max} \llvert \alpha
 \rrvert -2d_{\max})$
of being unpaired in $\cG|_{\alpha^c}$, and
so forth. The process
continues for at most $\ell-1$ steps, so that the probability of a
ball being unpaired in $\cG|_{\alpha^c}$ never exceeds
%
\begin{equation}
\label{9} \frac{s(\alpha)}{m- \llvert \alpha \rrvert  d_{\max}-d_{\max
}(\ell-1)} \leq \frac
{2s(\a)}{m} \leq \frac{2d_{\max}\llvert\alpha\rrvert}{m}
\end{equation}
if $m>2d_{\max}( \llvert \alpha \rrvert +\ell-1)$. Hence,
whenever $v\notin \alpha$,
\begin{equation*}
\IE\bigl\{ \llvert \xi_v\cap\alpha \rrvert \big|\a ,s(\alpha )\bigr
\} \leq\frac
{2d_{\max}\llvert\alpha\rrvert(\ell-1)}{m},
\end{equation*}
and \eqref{6.5} follows; bounding probabilities by expectations, \eqref{6} also follows.\\

In order to bound \eqref{7.5}, we proceed in a similar manner. The bound
is immediate if $v\in\a$, since $\llvert \eta_v\cap\alpha \rrvert \leq \llvert\xi_v\rrvert + 2K_v\leq \ell+2(\ell+6)$ on the event $A$;
so we now suppose that $v\notin\alpha$. We have $\llvert\eta_v\cap\alpha\rrvert = \llvert\xi_v\cap\alpha \rrvert +  \llvert (\eta_v\setminus\xi_v)\cap\alpha\rrvert$. The expectation of $\llvert\xi_v\cap\alpha \rrvert$ can be bounded by \eqref{6.5}, so we only need to consider the expectation of the latter summand on the event $A$. So,
assume that $\cT_\ell(v)$ has been explored. We now repeatedly
perform the swapping of paired balls $K_v$ times. Each time that a
swapping is performed, we could pick
a colour from $\alpha$ in one of two ways: either we pick a ball
from $\cG|_{\alpha^c}$ which is unpaired
in $\cG|_{\alpha^c}$, or we pick a ball directly from $\cG|_{\alpha
}$. The probability of the former is
no greater than
\begin{equation*}
\frac{s(\alpha)}{m-d_{\max}( \llvert \alpha \rrvert +\ell
)}\leq\frac{2s(\a)}{m} \leq \frac{2d_{\max}\llvert\alpha\rrvert}{m},
\end{equation*}
whenever $m>2d_{\max}( \llvert \alpha \rrvert +\ell)$ (note
that more and more
pairs are being put back now, so that
the denominator is in fact increasing), and the probability of the
latter is no greater than
\begin{equation*}
\frac{ d_{\max}\llvert \alpha \rrvert}{m-d_{\max}( \llvert \alpha \rrvert +\ell)} \leq\frac{2d_{\max}\llvert \alpha \rrvert}{m}.
\end{equation*}
Since, on the event $A$, we have $K_v \leq\ell+ 6$, we deduce that
the expected number of balls from $\cG|_\alpha$
being reached is no greater than
\begin{equation*}
\frac{4d_{\max}\llvert \alpha \rrvert K_v}{m} \leq\frac
{4d_{\max}\llvert \alpha \rrvert (\ell+6)}{m},
\end{equation*}
from which, when adding the last term of \eqref{6.5}, \eqref{7.5} follows. Using once again 
expectations to bound probabilities, \eqref{7}
also follows.
\end{proof}


\begin{corollary} \label{cor-rem}
Under the conditions of Lemma~\ref{lem3}, and assuming that $d_{\max}
\geq2$, that~$\ell\geq12$
and that $m \geq\max\{n,2d_{\max}(|\a|+\ell)\}$, we have
%
\begin{align}
\sum_{v=1}^n \IP\bigl[\xi_v
\cap\alpha\neq\emptyset\big|\alpha ,s(\alpha)\bigr] & \leq \llvert \alpha
\rrvert + \frac{2d_{\max}\llvert \alpha\rrvert(\ell-1) n}{m} \leq 2 \llvert \alpha \rrvert d_{\max}\ell;
\label{11}
\\
\sum_{v=1}^n \IE\bigl\{ \llvert
\xi_v\cap\alpha \rrvert \big| \alpha ,s(\alpha)\bigr\} & \leq
\ell\llvert \alpha
\rrvert + \frac{2d_{\max}\llvert \alpha\rrvert(\ell-1) n}{m} \leq 3 \llvert \alpha \rrvert d_{\max}\ell; \label{11.5}
\\
\label{12} \sum_{v=1}^n
\IP\bigl[\eta_v\cap\alpha\neq\emptyset,A\big| \alpha ,s(\alpha)
\bigr] & \leq \llvert\alpha
\rrvert + \frac{2d_{\max}\llvert \alpha\rrvert(3\ell+11) n}{m} \leq 8 \llvert \alpha \rrvert d_{\max}\ell;\\
\label{12.5} \begin{split}\sum_{v=1}^n
\IE\bigl\{ \llvert \eta_v\cap\alpha \rrvert I_A \big|
\alpha ,s(\alpha)\bigr\} & \leq (3\ell+12)\llvert\alpha
\rrvert + \frac{2d_{\max}\llvert \alpha\rrvert(3\ell+11) n}{m}\\ &\leq 10 \llvert \alpha \rrvert d_{\max}\ell;\end{split}%
\end{align}
\end{corollary}

We also need the following Efron--Stein type variance bound; see \cite
{Chen2017}.

\begin{lemma}\label{lem4} Let $\pi$ be a uniform permutation
on $[m]$, let \mbox{$B=(B_1,\dots,\break B_n)$}
be a sequence of i.i.d. random elements taking values in some suitable space,
and let~$f(\pi,B)$ be a real-valued function. Let $\tau_1,\dots,\tau
_{m-1}$ be independent transpositions, also
independent of all else, where $\tau_j$ transposes $j$ and a randomly
chosen integer in the
set $\{j,\ldots,m\}$, let $B'$ be an independent copy of $B$, and
for $1\leq i\leq n$, let
$B^i = (B_1,\dots,B_{i-1},B_i',B_i,\dots,B_n)$.
Then
\bes{
&\Var f(\pi,B)\\
&\qquad \leq\frac{1}{2} \sum_{j=1}^{m-1}
\IE\bigl(f(\pi,B)-f(\pi\tau_j,B)\bigr)^2
+ \frac{1}2\sum_{i=1}^n\IE
\bigl(f(\pi,B)-f\bigl(\pi,B^i\bigr)\bigr)^2.
}
\end{lemma}

The appearance of the particular transpositions in the lemma comes from
using a well-known construction of
uniform random permutations; see \cite{Knuth1969}.

\subsection{Completing the proof of Theorem \protect\ref{THM1}}

Using these preliminary results, we can now bound the two terms
appearing in Theorem~\ref{thm4}.

\paragraph{Bounding $\boldsymbol{\IE|G\D^2|}$.} The bound
%
\begin{equation}
\label{G_v-def} \llvert G_v \rrvert = \frac{n}{\sigma_{d,h}} \bigl\llvert
h\bigl(\cT_\ell(v)\bigr) \bigr\rrvert \leq \frac{n \llVert  h  \rrVert }{\sigma_{d,h}}
\end{equation}
is straightforward. Now, it is easy to see that we can write
%
\begin{equation}
\label{D_v-def} \D_v = \frac{1}{\sigma_{d,h}}\sum
_{w\in Q_v}\bigl(h\bigl(\cT^v_\ell (w)
\bigr)-h\bigl(\cT_\ell(w)\bigr)\bigr),
\end{equation}
where
\begin{equation*}
Q_v:= \{w: \xi_w\cap\eta_v \neq\emptyset\}.
\end{equation*}
Put in words, $Q_v$ is the set of those colours whose truncated
components are \emph{potentially} affected when
changing the underlying graph from $\cG$ to $\cG'_v$. We emphasize
``potentially'' here since, even
if we have $\cT_\ell(w)\neq\cT^v_\ell(w)$, it is still possible that~$h(\cT_\ell
(w))= h(\cT^v_\ell(w))$.

Note first that we have the crude bound $|\D_v| \leq2n\sigma
_{d,h}^{-1} \llVert  h  \rrVert $, so that, using \eqref{gamma_n},
%
\begin{equation}
\label{GD^2-exception} \IE\bigl\{ \llvert G\rrvert\D^2
I_{A^c}\bigr\} \leq\frac{4n^3 \llVert  h  \rrVert ^3\g}{\sigma_{d,h}^{3}} \leq \frac{4 \llVert  h  \rrVert ^3d_{\max}^{16}\ell^{16}}{8!n^{4}\sigma_{d,h}^{3}}.
\end{equation}
For the main part, we observe that
%
\begin{equation}
\begin{split}\label{523} \IE\bigl\{\D_v^2
I_A\bigr\} &\leq\frac{4 \llVert  h  \rrVert ^2}{\sigma
_{d,h}^2}\IE\bigl\{ \llvert Q_v
\rrvert ^2 I_A\bigr\}
\\
& \leq\frac{4 \llVert  h  \rrVert ^2}{\sigma_{d,h}^2} \sum_{w=1}^n\sum
_{w'=1}^n \IP[\xi_{w'}\cap
\eta_v \neq \emptyset , \xi_{w}\cap\eta_v \neq
\emptyset,A]. \end{split} %
\end{equation}
Enlarging the set with which $\xi_{w'}$ is allowed to overlap, we have
%
\begin{equation}
\begin{split}\label{enlarging} &\IP[\xi_{w'}\cap
\eta_v \neq\emptyset,\xi_w\cap\eta_v \neq
\emptyset,A]
\\
&\qquad \leq\IE\bigl\{\I\bigl[\xi_{w'}\cap(\xi_w\cup
\eta_v)\neq\emptyset\bigr] \I [\xi_{w}\cap
\eta_v\neq\emptyset] I_A\bigr\}. \end{split} %
\end{equation}
Noting that $ \llvert \xi_w  \rrvert \leq\ell$ and $ \llvert \eta_v  \rrvert\leq \ell+2(\ell+6) \leq4\ell$,
the latter on the event $A$ and for $\ell\geq12$, we now apply \eqref
{11} to the right-hand side of \eqref{enlarging} twice, once
for $\alpha=\xi_w\cup\eta_v$ and once for $\alpha=\eta_v$. Hence,
the double sum on the right-hand side of \eqref{523} becomes
%
\begin{equation}
\begin{split}\label{enlarged-1} &\sum_{w=1}^n
\sum_{w'=1}^n \IP[\xi_w\cap
\eta_v \neq \emptyset, \xi_{w'}\cap\eta_v \neq
\emptyset, A]
\\
&\qquad \leq2d_{\max}\ell\sum_{w=1}^n
\IE\bigl\{ \llvert \xi_w\cup \eta_v \rrvert \I[
\xi_{w}\cap\eta_v\neq\emptyset] I_A\bigr\}
\\
&\qquad \leq 10 d_{\max}\ell^2\sum
_{w=1}^n\IE\bigl\{\IP\bigl[\xi _{w}\cap
\eta_v\neq\emptyset,A\big|\eta_v,s(\eta_v)
\bigr]\bigr\} \leq 80 d_{\max}^2\ell^4 \end{split}
\end{equation}
(this line of argument is frequently repeated in the calculations that follow).
Thus,
%
\begin{equation}
\label{13} \IE\bigl\{ \llvert Q_v \rrvert ^2
I_A\bigr\} \leq 80d_{\max}^2\ell^4,
\end{equation}
which yields
\begin{equation*}
\IE\bigl\{\llvert G\rrvert\D^2  I_A\bigr\} =
\frac{1}{n}\sum_{v=1}^n\IE\bigl\{
\llvert G_v\rrvert\D_v^2 
I_A\bigr\} \leq\frac{320 \llVert  h  \rrVert ^3d_{\max}^2\ell^4
n}{\sigma_{d,h}^3};
\end{equation*}
hence
\begin{equation*}
\IE \bigl\{\llvert G\rrvert\D^2 \bigr\} \leq\frac{ \llVert  h
 \rrVert ^3d_{\max}^2\ell^{10}
n}{\sigma_{d,h}^3} \biggl(\frac{320}{12^6}
+ \frac{4d_{\max}^{14}\ell^{6}}{ 8!n^5} \biggr) \leq\frac{\llVert  h  \rrVert ^3d_{\max}^2\ell^{10}
n}{4536\sigma_{d,h}^3},
\end{equation*}
if $\max\{d_{\max},\ell\} \leq n^{1/4}$.

\paragraph{Bounding $\boldsymbol{\Var\IE(G\D|W)}$.}
In order to bound $\Var\IE(G\D|W)$, note that we can generate the
random configuration $\cG$ by means of a
uniformly chosen random permutation $\pi$ of $[m]$ by pairing
balls $\pi(1)$ and $\pi(2)$, pairing
balls $\pi(3)$ and~$\pi(4)$ and so forth. But also note that the same
configuration can be generated by more
than one permutation: For example, both $(\pi(2),\pi(1),\pi(3),\dots
,\pi(m))$
and~$(\pi(3),\pi(4),\pi(1),\pi(2),\pi(5),\dots,\pi(m))$
represent the same graph. Moreover, recall that
in the construction of $\cG_v'$, we encode the choices made in
Algorithm A by~$B_1,\dots,B_n$. With this
in mind, we define
\begin{equation*}
X_v^{\pi} = h\bigl(\cT_\ell(v)\bigr),\qquad
X_w^{\pi,B_v} = h\bigl(\cT ^v_\ell(w)
\bigr),
\end{equation*}
and we write $Q_v^{\pi,B_v}$ instead of $Q_v$ (and we do the same
with other quantities) to make the
dependence on $\pi$ and $B_v$ explicit. Now, with $B=(B_1,\dots
,B_n)$, write
\begin{equation*}
\IE(G\D|\pi,B) = \frac{1}{\sigma_{d,h}^2}\sum_{v=1}^n
X_v^{\pi} \sum_{w\in Q_v^{\pi,B_v}}
\bigl(X_w^{\pi}-X_w^{\pi,B_v}\bigr) =:
\frac{1}{\sigma_{d,h}^2} f(\pi,B),
\end{equation*}
so that we can apply Lemma~\ref{lem4}.
For $1\leq i<j\leq m$, let $\pi_{ij} = \pi\tau_{ij}$, where $\tau
_{ij}$ is the transposition switching $i$ and $j$.
Now, it is clear that
%
\begin{equation}
\label{14} \IE\bigl(f(\pi,B)-f(\pi_{ij},B)\bigr)^2 =
\IE\bigl(f(\pi,B)-f(\pi_{13},B)\bigr)^2
\end{equation}
(unless $i=2k+1$ and $j=2k+2$ for some $0\leq k<m$, in which case the
expectation on the left-hand side vanishes,
since the transposition does not affect the underlying graph). Hence,
it is enough to bound the right-hand side
of \eqref{14}. Adding and subtracting corresponding terms, we have
\begin{equation*}
\begin{split} & f(\pi,B) - f(\pi_{13},B)
\\
&\qquad = \sum_{v=1}^n
X_v^{\pi} \sum_{w\in Q_v^{\pi,B_v}}
\bigl(X_w^{\pi}-X_w^{\pi,B_v}\bigr) - \sum
_{v=1}^n X_v^{\pi_{13}}
\sum_{w\in Q_v^{\pi_{13},B_v}}\bigl(X_w^{\pi_{13}}-X_w^{\pi_{13},B_v}
\bigr)
\\
&\qquad = \sum_{v=1}^n
\bigl(X_v^{\pi} - X_v^{\pi_{13}}\bigr) \sum
_{w\in Q_v^{\pi,B_v}}\bigl(X_w^{\pi}-X_w^{\pi,B_v}
\bigr)
\\
&\qquad\quad{} + \sum_{v=1}^n
X_v^{\pi_{13}} \biggl(\sum_{w\in Q_v^{\pi,B_v}}
\bigl(X_w^{\pi}-X_w^{\pi,B_v} \bigr)-\sum
_{w\in
Q_v^{\pi_{13},B_v}}\bigl(X_w^{\pi}-X_w^{\pi,B_v}
\bigr)\biggr)
\\
&\qquad\quad{} + \sum_{v=1}^n
X_v^{\pi_{13}} \sum_{w\in Q_v^{\pi_{13},B_v}}
\bigl(X_w^{\pi}-X_w^{\pi_{13}}\bigr)
\\
&\qquad\quad{} - \sum_{v=1}^n
X_v^{\pi_{13}} \sum_{w\in Q_v^{\pi_{13},B_v}}
\bigl(X_w^{\pi,B_v}-X_w^{\pi_{13},B_v}\bigr).
\end{split} %
\end{equation*}
Letting $\chi$ be the set of colours of the balls $\pi(1)$, $\pi
(2)$, $\pi(3)$ and $\pi(4)$,
we obtain
%
\begin{equation}
\label{f-squared-bound} \IE \bigl\{\bigl(f(\pi,B) - f(\pi_{13},B)
\bigr)^2 I_A \bigr\} \leq4 \llVert h \rrVert
^4\IE\bigl\{\bigl(R_1^2+R_2^2+R_3^2+R_4^2
\bigr) I_A\bigr\},
\end{equation}
where the event $A$ is as in \eqref{A_n} and where
\begin{align*}
R_1 & = \sum_{v=1}^n \I
\bigl[\chi\cap\xi_v^\pi\neq\emptyset\bigr] \bigl\llvert
Q_v^{\pi,B_v} \bigr\rrvert ,
\\
R_2 & = \sum_{v=1}^n \bigl
\llvert Q_v^{\pi,B_v} \setdiff Q_v^{\pi
_{13},B_v}
\bigr\rrvert ,
\\
R_3 & = \sum_{v=1}^n \sum
_{w\in Q_v^{\pi,B_v}} \I\bigl[\chi\cap\xi_w^\pi
\neq\emptyset\bigr],
\\
R_4 & = \sum_{v=1}^n \sum
_{w\in Q_v^{\pi,B_v}} \I\bigl[\cT^{v,\pi,B_v}_\ell(w)
\neq\cT^{v,\pi_{13},B_v}_\ell(w)\bigr],
\end{align*}
(note that for $R_3$ and $R_4$, we have replaced $\pi_{13}$ by $\pi$
and vice versa, since the two random permutations
are exchangeable). We now proceed to bound the four error terms individually.
In order to keep the formulae short, we abbreviate multiple sums such
as $\sum_{v=1}^n\sum_{w=1}^n$ to $\sum_{v,w}$,
where it is understood that summation always ranges from $1$ to $n$.

We begin with some preliminary calculations involving $\chi$.
First, let $\cG$ be the configuration generated by $\pi$ as before,
and let the set of colours $\alpha$ be a
function of $\cG$ (but not of $\pi$ directly).
Let $E_1(\a)$ denote the set of edges incident to $\a$, and let
$\cE:=  \{ \{\pi(1), \pi(2)\}, \{\pi(3),\pi(4)\}  \}$.
Then $\chi\cap\a\neq\emptyset$ if $\cE\cap E_1(\a) \neq
\emptyset$. Hence, since a given edge has probability
$1/(m/2)$ of being represented by $\{\pi(2i-1),\pi(2i)\}$ for
any $1\leq i \leq m/2$,
and since $ \llvert  E_1(\a)  \rrvert \leq d_{\max} \llvert
\a \rrvert $, we have
%
\begin{equation}
\label{chi-with-alpha} \IP[\chi\cap\alpha\neq\emptyset|\cG] \leq2 \frac{ \llvert \alpha \rrvert  d_{\max}}{m/2}
= \frac{4 \llvert \alpha \rrvert  d_{\max}}{m}.
\end{equation}
We also need to bound probabilities of the form
\begin{equation*}
\IP[\chi\cap\a_1\neq\emptyset, \chi\cap\alpha_2\neq
\emptyset |\cG],
\end{equation*}
for pairs of colour sets $\a_1,\a_2$ as above.
To this end, let $E_2(\a_1,\a_2)$ denote the set of
edges in $\cG$ that join a vertex in $\a_1$ to one in $\a_2$, and define
$E(\a_1,\a_2):= E_1(\a_1\cap\a_2) \cup E_2(\a_1,\a_2)$.
Then note that
%
\begin{equation}
\begin{split}\label{split} &\I[\chi\cap\alpha_1\neq
\emptyset, \chi\cap\alpha_2\neq \emptyset]
\\
&\qquad \leq \I\bigl[\cE\subset E_1(\a_1 \cup
\a_2)\bigr] + \I\bigl[\cE\cap E(\a_1,\a _2)
\neq\emptyset\bigr]. \end{split} %
\end{equation}
Now it is easy to see that
%
\begin{equation}
\label{both-in} \IP\bigl[\cE\subset E_1(\a_1 \cup
\a_2)\big|\cG\bigr] \leq\frac{(d_{\max}
 \llvert \a_1 \cup\a_2  \rrvert )^2}{(m/2)(m/2-1)} \leq\frac{5d_{\max}^2 \llvert \a_1 \cup\a_2  \rrvert ^2}{m^2},
\end{equation}
and that
%
\begin{equation}
\begin{split}\label{one-in} \IP\bigl[\cE\cap E(\a_1,
\a_2)\neq\emptyset\big|\cG\bigr] &\leq\IE\bigl\{ \bigl\llvert \cE\cap
E(\a_1,\a_2) \bigr\rrvert \big|\cG\bigr\}
\\
&\leq\frac{4}{m}\bigl( \bigl\llvert E_1(\a_1
\cap\a_2) \bigr\rrvert + \bigl\llvert E_2(
\a_1,\a_2) \bigr\rrvert \bigr) . \end{split} %
\end{equation}
Note, in particular, that \eqref{11.5} implies that, for any $1\leq
v\leq n$,
%
\begin{equation}
\label{E_1-xi} \sum_{w=1}^n \IE \bigl
\llvert E_1(\xi_w \cap\xi_v) \bigr\rrvert
\leq d_{\max}\sum_{w=1}^n \IE
\llvert \xi_w\cap\xi_v \rrvert \leq3d_{\max
}^2
\ell^2;
\end{equation}
then, taking $\a$ to be the set of all colours joined to $\xi_v$
in $\cG$, so that $|\a| \leq d_{\max}\ell$,
\eqref{11.5} also implies that
%
\begin{equation}
\label{E_2-xi} \sum_{w=1}^n \IE \bigl
\llvert E_2(\xi_w,\xi_v) \bigr\rrvert \leq
3d_{\max}^2\ell^2.
\end{equation}
In similar fashion, using \eqref{12.5}, we also have
%
\begin{equation}
\sum_{w=1}^n \IE\llvert
E_1(\eta_w\cap\eta_v) \rrvert  \leq 40d_{\max
}^2\ell^2,
\qquad
\sum_{w=1}^n \IE\llvert E_2(\eta_w,\eta_v) \rrvert   \leq 40d_{\max
}^2\ell^2. \label{E_12-eta}
\end{equation}

\noindent\emph{Bound on ${\IE\{ R_1^2 I_A\}}$.\enskip} First, write
\begin{equation*}
R_1 = \sum_{v,w} \I\bigl[\chi\cap
\xi_v^\pi\neq\emptyset\bigr]\I\bigl[\xi^\pi
_w\cap\eta_v^{\pi,B_v}\neq\emptyset\bigr].
\end{equation*}
Then, using \eqref{split}--\eqref{one-in}, we have
\begin{equation*}
\begin{split} &\IE\bigl\{R_1^2
I_A\bigr\}
\\
&\qquad = \IE\sum_{v,v',w,w'} \I\bigl[\chi\cap
\xi_{v'}^\pi \neq\emptyset\bigr]\I\bigl[\chi \cap \xi_{v}^\pi
\neq\emptyset\bigr] \I\bigl[\xi^\pi_{w'}\cap
\eta_{v'}^{\pi,B_{v'}}\neq\emptyset\bigr]
\\
&\qquad\qquad\qquad\qquad\quad{}\times \I\bigl[\xi^\pi_{w}\cap
\eta_v^{\pi,B_v}\neq\emptyset\bigr] I_A
\\
&\qquad \leq\frac{20\ell^2d_{\max}^2}{m^2}\IE\sum_{v,v',w,w'} \I[
\xi_{w'}\cap\eta_{v'}\neq\emptyset] \I[\xi_{w}\cap
\eta_v\neq\emptyset] I_A
\\
&\qquad\quad{} + \frac{4}m \IE\sum_{v,v',w,w'}
\bigl( \bigl\llvert E_1\bigl(\xi_{v'}^\pi\cap
\xi_{v}^\pi\bigr) \bigr\rrvert + \bigl\llvert
E_2\bigl(\xi_{v'}^\pi,\xi_{v}^\pi
\bigr) \bigr\rrvert \bigr) \I[\xi_{w'}\cap\eta_{v'}\neq
\emptyset]
\\
&\qquad\qquad\qquad\qquad\quad{}\times \I[\xi_{w}\cap\eta_v\neq\emptyset]
I_A. \end{split} %
\end{equation*}
Much as for \eqref{enlarged-1}, we use inequality analogous to \eqref
{enlarging} to show that the first term yields at most
\begin{equation*}
\begin{split} \frac{20d_{\max}^2\ell^2}{m^2} \sum_{v,v'}
144 d_{\max}^2\ell^4 & \leq\frac{2880n d_{\max}^4\ell^6 }{m},
\end{split} %
\end{equation*}
where we also used that $n\leq m$.
For the second term, again using inequalities in the spirit of \eqref
{enlarging}, we have
\begin{equation*}
\begin{split} &\frac{4}m\IE\sum
_{v,v',w,w'} \bigl( \bigl\llvert E_1\bigl(
\xi_{v'}^\pi\cap \xi_{v}^\pi\bigr)
\bigr\rrvert + \bigl\llvert E_2\bigl(\xi_{v'}^\pi,
\xi _{v}^\pi\bigr) \bigr\rrvert \bigr) \I[
\xi_{w'}\cap\eta_{v'}\neq\emptyset] \I[\xi_{w}\cap
\eta_v\neq\emptyset] I_A
\\
&\qquad \leq\frac{4}m \sum_{v,v'} 144
d_{\max}^2\ell^4 \IE\bigl\{ \bigl\llvert
E_1\bigl(\xi_{v'}^\pi\cap\xi_{v}^\pi
\bigr) \bigr\rrvert + \bigl\llvert E_2\bigl(\xi_{v'}^\pi,
\xi_{v}^\pi\bigr) \bigr\rrvert \bigr\} \leq
\frac{3456 nd_{\max}^4\ell^6 }{m}, \end{split} %
\end{equation*}
where the last line follows from \eqref{E_1-xi} and \eqref{E_2-xi},
and the two bounds combine to give
%
\begin{equation}
\label{R_1-bnd} \IE\bigl\{R_1^2 I_A\bigr\}
\leq\frac{6336n d_{\max}^4\ell^6 }{m}.
\end{equation}

\noindent\emph{Bound on ${\IE\{ R_2^2 I_A\}}$.\enskip}
First, note that
\begin{equation*}
\begin{split} & \I\bigl[w\in Q_v^{\pi,B_v}\setdiff
Q_v^{\pi_{13},B_v}\bigr]
\\
&\qquad \leq\bigl(\I\bigl[\chi\cap\xi_w^\pi\neq\emptyset
\bigr] + \I\bigl[\chi \cap\eta_v^{\pi,B_v}\neq\emptyset\bigr]
\bigr)\I\bigl[\xi_w^\pi\cap\eta _v^{\pi,B_v}
\neq\emptyset\bigr]
\\
&\qquad\quad{} + \bigl(\I\bigl[\chi\cap\xi_w^{\pi_{13}}\neq
\emptyset\bigr]+ \I\bigl[\chi\cap\eta_v^{\pi_{13},B_v}\neq\emptyset
\bigr]\bigr) \I\bigl[\xi_w^{\pi_{13}}\cap\eta_v^{\pi_{13},B_v}
\neq\emptyset\bigr]. \end{split} %
\end{equation*}
Hence, using the inequality $(a_1+\cdots+a_k)^2\leq k(a_1^2+\cdots
+a_k^2)$ and the exchangeability of $\pi$ and $\pi_{13}$,
%
\begin{align}
&\IE\bigl\{R^2_2 I_A\bigr\}
\nonumber
\\
&\qquad \leq\IE\biggl(\sum_{v,w} I_A
\Bigl[\bigl(\I\bigl[\chi\cap\xi_w^\pi\neq \emptyset\bigr]+ \I\bigl[\chi\cap\eta_v^{\pi
,B_v}\neq\emptyset
\bigr]\bigr)\nonumber\\[-1.5ex]
&\quad\qquad\qquad\qquad\qquad\qquad\qquad\qquad{}\times\I\bigl[\xi_w^\pi\cap\eta_v^{\pi,B_v}
\neq \emptyset\bigr]
\nonumber
\\[1.5ex]
&\qquad\qquad\qquad\quad{} +\bigl(\I\bigl[\chi\cap\xi_w^{\pi
_{13}}\neq
\emptyset\bigr]+ \I\bigl[\chi\cap \eta_v^{\pi_{13},B_v}\neq\emptyset
\bigr]\bigr)\nonumber\\
&\quad\qquad\qquad\qquad\qquad\qquad\qquad\qquad{}\times \I\bigl[\xi_w^{\pi_{13}}\cap \eta_v^{\pi_{13},B_v}
\neq\emptyset \bigr] \Bigr]\biggr)^2
\nonumber
\\
&\qquad \leq8 \IE\biggl(\sum_{v,w} \I\bigl[\chi\cap
\xi_w^\pi\neq \emptyset\bigr] \I\bigl[
\xi^\pi_w\cap\eta^{\pi,B_v}_v\neq
\emptyset\bigr] I_A\biggr)^2
\nonumber
\\
& \qquad\quad{}+ 8 \IE\biggl(\sum_{v,w} \I\bigl[\chi
\cap\eta_v^{\pi
,B_v}\neq\emptyset\bigr] \I\bigl[
\xi^\pi_w\cap\eta^{\pi,B_v}_v\neq
\emptyset\bigr] I_A\biggr)^2
\nonumber
\\
\begin{split} & \qquad= 8 \IE\sum_{v,v',w,w'}
\I\bigl[\chi\cap\xi_{w'}^\pi\neq\emptyset\bigr]\I\bigl[\chi
\cap \xi_{w}^\pi \neq\emptyset\bigr]
\\[-1ex]
&\qquad\qquad\qquad\qquad{} \times \I\bigl[\xi^\pi_w\cap
\eta^{\pi,B_v}_v\neq\emptyset\bigr] \I\bigl[\xi^\pi_{w'}
\cap\eta^{\pi,B_{v'}}_{v'}\neq\emptyset\bigr] I_A
\end{split} %
\label{R_2-part1}
\\
\begin{split} & \qquad\quad{}+ 8 \IE\sum
_{v,v',w,w'} \I\bigl[\chi\cap\eta_{v'}^{\pi,B_{v'}} \neq
\emptyset\bigr]\I\bigl[\chi\cap \eta_{v}^{\pi,B_{v}} \neq\emptyset
\bigr]
\\[-1ex]
& \qquad\qquad\qquad\qquad\quad{} \times \I\bigl[\xi^\pi_w\cap
\eta^{\pi,B_v}_v\neq\emptyset\bigr] \I\bigl[\xi^\pi_{w'}
\cap\eta^{\pi,B_{v'}}_{v'}\neq\emptyset\bigr] I_A.
\end{split} %
\label{R_2-part2}
\end{align}
We now use \eqref{split}--\eqref{one-in}, together with inequalities
such as in \eqref{enlarging}, to give
\begin{align*}
&\sum_{v,v'}\IE\bigl\{\I\bigl[\chi\cap
\xi_{w'}^\pi\neq \emptyset\bigr]\I \bigl[\chi\cap
\xi_{w}^\pi\neq\emptyset\bigr] \I\bigl[\xi^\pi_w
\cap\eta^{\pi
,B_v}_v\neq\emptyset\bigr] \I\bigl[
\xi^\pi_{w'}\cap\eta^{\pi
,B_{v'}}_{v'}\neq
\emptyset\bigr] I_A\bigr\}
\\*
&\qquad \leq\sum_{v,v'}\IE\biggl\{\biggl(
\frac{20d_{\max}^2\ell ^2}{m^2} + \frac{4}m\bigl( \bigl\llvert E_1(
\xi_w\cap\xi_{w'}) \bigr\rrvert + \bigl\llvert
E_2(\xi_w,\xi_{w'}) \bigr\rrvert \bigr)
\biggr)
\\
&\qquad\qquad\qquad{} \times \I\bigl[\eta_v \cap(\xi_w\cup
\xi_{w'}\cup \eta_{v'}) \neq\emptyset \bigr] \I\bigl[
\eta_{v'} \cap(\xi_w\cup\xi_{w'}) \neq\emptyset
\bigr] I_A \biggr\}
\\
&\qquad \leq\IE\biggl\{\frac{20d_{\max}^2\ell^2}{m^2} + \frac{4}m\bigl( \bigl
\llvert E_1(\xi_w\cap\xi_{w'}) \bigr\rrvert +
\bigl\llvert E_2(\xi _w,\xi_{w'}) \bigr\rrvert
\bigr) \biggr\} 768 d_{\max}^2\ell^4, %
\end{align*}
this last using \eqref{12} twice. Now sum over $w$ and $w'$, using
\eqref{E_1-xi} and \eqref{E_2-xi},
to give a contribution to $\IE\{R_2^2 I_A\}$ from \eqref{R_2-part1}
of at most
\begin{equation*}
\frac{8n}{m}\bigl(20d_{\max}^2\ell^2 +
12d_{\max}^2\ell^2 + 12d_{\max}
^2\ell^2\bigr) 768 d_{\max}^2
\ell^4 = \frac{270{,}336n d_{\max}^4\ell^6}{m}.
\end{equation*}
The term in \eqref{R_2-part2} is treated analogously, using \eqref
{E_12-eta} in place of
\eqref{E_1-xi} and \eqref{E_2-xi}, giving a further
\begin{equation*}
\frac{8n}{m}(320 + 160 + 160) 144 d_{\max}^4
\ell^6 = \frac{737{,}280nd_{\max}^4\ell^6}{m},
\end{equation*}
so that
%
\begin{equation}
\label{R_2-bnd} \IE\bigl\{R_2^2 I_A\bigr\}
\leq\frac{1{,}007{,}616nd_{\max}^4\ell^6}{m}.
\end{equation}

\noindent\emph{Bound on ${\IE\{R_3^2 I_A \}}$.\enskip} For this
term, we have
\begin{align*}
&\IE\bigl\{R_3^2 I_A\bigr\}
\\
&\quad = \IE\Biggl(\sum_{v=1}^n\sum
_{w=1}^n \I\bigl[\chi\cap
\xi_{w}^\pi \neq\emptyset\bigr] \I\bigl[
\xi_w^\pi\cap\eta_v^{\pi,B_v}\neq
\emptyset\bigr] I_A \Biggr)^2
\\
&\quad \leq\IE\sum_{v,v',w,w'} \I[\chi\cap
\xi_{w}\neq\emptyset]\I[\chi\cap\xi _{w'}\neq \emptyset] \I[
\xi_w\cap\eta_v\neq\emptyset] \I[\xi_{w'}\cap
\eta_{v'}\neq\emptyset] I_A. %
\end{align*}
This is the term in \eqref{R_2-part1}, but without the factor of $8$, giving
%
\begin{equation}
\label{R_3-bnd} \IE\bigl\{ R_3^2 I_A\bigr\}
\leq\frac{33{,}792nd_{\max}^4\ell^6}{m}.
\end{equation}

\noindent\emph{Bound on ${\IE\{R_4^2 I_A \}}$.\enskip} In order to
bound $\IE\{R_4^2 I_A\}$, note that
\begin{align*}
& \I\bigl[w\in Q_v^{\pi,B_v}\bigr]\I\bigl[
\xi^{\prime\pi,B_v}_{v,w} \neq \xi ^{\prime\pi_{13},B_v}_{v,w}\bigr]
\\
&\qquad \leq \I\bigl[\xi_w^\pi\cap\eta^{\pi,B_v}_v
\neq\emptyset\bigr] \bigl(\I\bigl[\chi\cap\xi_w^{\pi}\neq
\emptyset\bigr] + \I\bigl[\chi\cap\eta _v^{\pi,B_v}\neq\emptyset
\bigr]\bigr), %
\end{align*}
so that
\begin{align*}
\IE\bigl\{R_4^2 I_A\bigr\}
\leq{}&\IE\biggl(\sum_{v,w} \I\bigl[
\xi_w^\pi\cap\eta^{\pi,B_v}_v\neq
\emptyset\bigr] I_A \bigl(\I\bigl[\chi\cap\xi_w^{\pi}
\neq\emptyset\bigr] + \I\bigl[\chi \cap\eta _v^{\pi,B_v}\neq
\emptyset\bigr]\bigr)\biggr)^2
\\
\leq{}&4\IE\biggl(\sum_{v,w} \I\bigl[
\xi_w^\pi\cap\eta^{\pi,B_v}_v\neq
\emptyset\bigr] \I\bigl[\chi\cap\xi_w^{\pi}\neq\emptyset
\bigr] I_A\biggr)^2
\\
&{} + 4\IE\biggl(\sum_{v,w} \I\bigl[
\xi_w^\pi\cap\eta^{\pi,B_v}_v\neq
\emptyset\bigr] \I\bigl[\chi\cap\eta_v^{\pi,B_v}\neq\emptyset
\bigr] I_A\biggr)^2. %
\end{align*}
This is half the sum of the quantities given in \eqref{R_2-part1}
and \eqref{R_2-part2}, and hence
yields
%
\begin{equation}
\label{R_4-bnd} \IE\bigl\{R_4^2 I_A\bigr\}
\leq\frac{503{,}808nd_{\max}^4\ell^6}{m}.
\end{equation}
Substituting \eqref{R_1-bnd}, \eqref{R_2-bnd}, \eqref{R_3-bnd}
and \eqref{R_4-bnd} into \eqref{f-squared-bound} gives
\begin{equation*}
\IE \bigl\{\bigl(f(\pi,B) - f(\pi_{13},B)\bigr)^2
I_A \bigr\} \leq \frac{6{,}206{,}208 \llVert  h  \rrVert ^4n d_{\max}^4\ell^6}{m}
\leq \frac{ \llVert  h  \rrVert ^4n d_{\max}^4\ell^{16}}{9976m}.
\end{equation*}
The calculation on $A^c$ is based on the crude bound
\begin{equation*}
\bigl\llvert f(\pi,B) \bigr\rrvert \leq2n^2 \llVert h \rrVert
^2,
\end{equation*}
together with \eqref{gamma_n}, giving
\begin{align*}
\IE \bigl\{\bigl(f(\pi,B) - f(\pi_{13},B)
\bigr)^2 \I_{A^c} \bigr\} & \leq 16n^4 \llVert h
\rrVert ^4 \g
\\
& \leq\frac{16 n^4  \llVert  h  \rrVert ^4d_{\max}^{16}\ell
^{16}}{8!m^7}
\\
&\leq\frac{\llVert  h  \rrVert ^4nd_{\max}^{4}\ell
^{16}}{2520m}\times\frac
{d_{\max}^{12}}{m^2n}. %
\end{align*}
Thus, for $d_{\max}\leq n^{1/4}$ and $m\geq n$, we have
%
\begin{equation}
\label{pi13-term} \IE\bigl(f(\pi,B) - f(\pi_{13},B)\bigr)^2
\leq \frac{ \llVert  h  \rrVert ^4 nd_{\max}^4\ell^{16}}{2011m}.
\end{equation}
For the second sum in Lemma~\ref{lem4}, we have
\begin{equation*}
f(\pi,B) - f\bigl(\pi,B^v\bigr) = X_v^{\pi}
\biggl(\sum_{w\in Q_v^{\pi,B_v}}\bigl(X_w^{\pi}-X_w^{\pi,B_v}
\bigr) - \sum_{w\in Q_v^{\pi,B'_v}}\bigl(X_w^{\pi}-X_w^{\pi,B'_v}
\bigr)\biggr),
\end{equation*}
so that
\begin{equation*}
\bigl\llvert f(\pi,B) - f\bigl(\pi,B^v\bigr) \bigr\rrvert \leq2
\llVert h \rrVert ^2\bigl( \bigl\llvert Q_v^{\pi,B_v}
\bigr\rrvert + \bigl\llvert Q_v^{\pi,B'_v} \bigr\rrvert \bigr).
\end{equation*}
Hence, using exchangeability to replace $B_v'$ by $B_v$, we deduce that
\begin{equation*}
\IE\bigl\{\bigl(f(\pi,B) - f\bigl(\pi,B^v\bigr)\bigr)^2
I_A \bigr\} \leq16 \llVert h \rrVert ^4\IE \bigl\{
\llvert Q_v \rrvert ^2 I_A \bigr\},
\end{equation*}
and, using \eqref{13}, this gives
\begin{equation*}
\IE\bigl\{\bigl(f(\pi,B) - f\bigl(\pi,B^v\bigr)\bigr)^2
I_A \bigr\} \leq16 \llVert h \rrVert ^4
80d_{\max}^2\ell^4 \leq \frac{320}{12^{12}}\llVert h \rrVert ^4d_{\max}^4\ell^{16},
\end{equation*}
where we used that $d_{\max}\geq 2$ and $\ell\geq 12$.
Then, since $|f(\pi,B) - f(\pi,B^v)| \leq4n \llVert  h  \rrVert ^2$, it is
immediate that
\begin{equation*}\begin{split}
\IE\bigl\{\bigl(f(\pi,B) - f\bigl(\pi,B^v\bigr)\bigr)^2
I_{A^c} \bigr\} {}&{} \leq16 n^2 \llVert h \rrVert ^4
\g {}\leq \frac{16n^2\llVert h \rrVert ^4 d_{\max}^{16}\ell^{16}}{8!m^7} \\
&{}\leq \frac{16}{8!} \llVert h \rrVert ^4 d_{\max}^4
\ell^8\times\frac{d_{\max}^{12}\ell^8}{m^5}\leq\frac{16}{8!12^8} \llVert h \rrVert ^4 d_{\max}^4
\ell^{16},
\end{split}
\end{equation*}
where we used that $\max\{d_{\max},\ell\}\leq n^{1/4}$.
Hence,
%
\begin{equation}
\label{B_v-term} \IE\bigl(f(\pi,B) - f\bigl(\pi,B^v\bigr)
\bigr)^2 \leq\frac{1}{10^{10}} \llVert h \rrVert ^4 d_{\max}
^4\ell^{16}.
\end{equation}
Substituting these bounds into Lemma~\ref{lem4}, we obtain that
\begin{equation*}
\Var f(\pi,B) \leq \frac{1}{4021} n \llVert h \rrVert ^4
d_{\max
}^4\ell^{16},
\end{equation*}
and hence that
\begin{equation*}
\sqrt{\Var\IE(G\D| W)} \leq\frac{\sqrt n  \llVert  h  \rrVert ^2 d_{\max}^2\ell
^8}{63 \sigma_{d,h}^2},
\end{equation*}
completing the proof of Theorem~\ref{THM1}.

\subsection{The variance $\s^2_{d,h}$}\label{sect-variance}
It follows, by substituting $f(w)=1$ and then $f(w)=w$ for all $w$ into
the Stein coupling \eqref{3},
that $\Var W = \IE\{G\D\}$. Recalling the definitions \eqref
{G_v-def} and \eqref{D_v-def} of $G_v$ and $\D_v$,
it then follows that
%
\begin{equation}
\label{variance} \s^2_{d,h} = -n\IE \biggl\{ h\bigl(
\cT_\ell(I)\bigr)\sum_{w \in Q_I}\bigl(h\bigl(
\cT ^I_\ell(w)\bigr) - h\bigl(\cT_\ell(w)\bigr)
\bigr) \biggr\},
\end{equation}
where $I$ denotes a randomly chosen vertex in $[n]$. Under asymptotic
circumstances in which
the expectation in \eqref{variance} remains of order $\bigo(1)$
as $n\toinf$, this yields a variance $\s^2_{d,h}$
of order $\bigo(n)$. Broadly speaking, such circumstances are those in
which the value of $h(\cT_\ell(v))$ is
not much influenced by vertices far from $v$.
As far as the accuracy in Theorem~\ref{THM1} is concerned, it is
advantageous to have $n^{-1}\s^2_{d,h}$
bounded below as $n\toinf$. This is equivalent to requiring that the
expectation
in~\eqref{variance} does not tend to zero as $n\toinf$, which might
usually be supposed to be the case.
If, however, $h(\cT_\ell(v)):= \I[d_v = k]$ for some $k$,
then $h(\cT^v_\ell(w)) = h(\cT_\ell(w))$ for all $v,w$,
and the expectation would be exactly zero---as it has to be, since the
number of vertices of any
given degree $k$ is fixed in the model. So, in practice, this condition
has to be checked.

\section*{Acknowledgements}

We thank the referees for their careful reading, helpful comments and additional references. We also thank Siva Athreya and D. Yogeshwaran for sharing an early draft of their manuscript with us. This work was started while ADB was Saw Swee Hock Professor of Statistics at the National University of Singapore. ADB thanks the Department of Statistics and Applied Probability at the National University of Singapore and the mathematics departments of the University of Melbourne and Monash University for their kind hospitality. ADB was also supported in part by Australian Research Council Grants Nos DP120102728, DP120102398, DP150101459 and DP150103588, and by their Centre of Excellence for Mathematical and Statistical Frontiers. AR~was supported in part by NUS Research Grant R-155-000-167-112.

\setlength{\bibsep}{0.5ex}
\def\bibfont{\small}


\end{document}